\documentclass[a4paper,oneside,11pt]{article}

\usepackage{geometry}                
\usepackage{graphicx}
\usepackage{amsmath,amsfonts,amssymb,amsthm}
\usepackage{epstopdf}
\usepackage[utf8]{inputenc} 
\usepackage[T1]{fontenc} 
 
\usepackage[english]{babel} 

\title{A short proof of Handel and Mosher's alternative for subgroups of $\text{Out}(F_N)$}
\author{Camille Horbez}

\usepackage[all]{xy}
\usepackage{bbm}
\begin{document}
\maketitle
\newtheorem{de}{Definition} [section]
\newtheorem{theo}[de]{Theorem} 
\newtheorem{prop}[de]{Proposition}
\newtheorem{lemma}[de]{Lemma}
\newtheorem{cor}[de]{Corollary}
\newtheorem{propd}[de]{Proposition-Definition}

\theoremstyle{remark}
\newtheorem{rk}[de]{Remark}
\newtheorem{ex}[de]{Example}
\newtheorem{question}[de]{Question}

\normalsize

\addtolength\topmargin{-.5in}
\addtolength\textheight{1.in}
\addtolength\oddsidemargin{-.045\textwidth}
\addtolength\textwidth{.09\textwidth}

\begin{abstract}
We give a short proof of a theorem of Handel and Mosher \cite{HM09} stating that any finitely generated subgroup of $\text{Out}(F_N)$ either contains a fully irreducible automorphism, or virtually fixes the conjugacy class of a proper free factor of $F_N$, and we extend their result to non finitely generated subgroups of $\text{Out}(F_N)$.
\end{abstract}

\section*{Introduction}

Let $N\ge 2$, and let $F_N$ denote a finitely generated free group of rank $N$. A \emph{free factor} of $F_N$ is a subgroup $A$ of $F_N$ such that $F_N$ splits as a free product of the form $F_N=A\ast B$, for some subgroup $B\subseteq F_N$. An automorphism $\Phi\in\text{Out}(F_N)$ is \emph{fully irreducible} if no power of $\Phi$ preserves the conjugacy class of any proper free factor of $F_N$. The goal of this paper is to give a short proof of the following classification theorem for subgroups of $\text{Out}(F_N)$, which was shown by Handel and Mosher in the case of finitely generated subgroups of $F_N$ in \cite{HM09}.

\begin{theo}\label{HM}
Every (possibly non finitely generated) subgroup of $\text{Out}(F_N)$ either 
\begin{itemize}
\item contains two fully irreducible elements that generate a rank two free subgroup, or
\item is virtually cyclic, generated by a fully irreducible automorphism, or
\item virtually fixes the conjugacy class of a proper free factor of $F_N$.
\end{itemize}
\end{theo}

Our proof of Theorem \ref{HM} involves studying the action of subgroups of $\text{Out}(F_N)$ on the free factor complex, whose hyperbolicity was proved by Bestvina and Feighn in \cite{BF12} (see also \cite{KR12} for an alternative proof) and whose Gromov boundary was described by Bestvina and Reynolds \cite{BR13} and Hamenstädt \cite{Ham12}. We also use elementary tools that originally arose in the study of random walks on groups, by studying stationary measures on the boundaries of outer space and of the free factor complex.

Theorem \ref{HM} has already found various applications, for example to the study of morphisms from lattices to $\text{Out}(F_N)$ \cite{BW11} or to spectral rigidity questions \cite{CFKM12}.   

Handel and Mosher have generalized Theorem \ref{HM} in a recent series of papers \cite{HM13-1,HM13-2,HM13-3,HM13-4,HM13-5} to give a complete classification of finitely generated subgroups of $\text{Out}(F_N)$, analogous to Ivanov's classification of subgroups of the mapping class group of a finite type oriented surface \cite{Iva92}.

\section*{Acknowledgments}

I warmly thank my advisor Vincent Guirardel for his numerous and helpful advice that led to significant improvements in the exposition of the proof.

\section{Review}

\subsection{Gromov hyperbolic spaces}

A geodesic metric space $(X,d)$ is \emph{Gromov hyperbolic} if there exists $\delta>0$ such that for all $x,y,z\in X$, and all geodesic segments $[x,y],[y,z]$ and $[x,z]$, we have $N_{\delta}([x,z])\subseteq N_{\delta}([x,y])\cup N_{\delta}([y,z])$ (where given a subset $Y\subseteq X$ and $r\in\mathbb{R}_+$, we denote by $N_r(Y)$ the $r$-neighborhood of $Y$ in $X$). The \emph{Gromov boundary} $\partial X$ of $X$ is the space of equivalence classes of quasi-geodesic rays in $X$, two rays being equivalent if their images lie at bounded Hausdorff distance. 

\paragraph*{Isometry groups of Gromov hyperbolic spaces.}

Let $X$ be a hyperbolic geodesic metric space. An isometry $\phi$ of $X$ is \emph{loxodromic} if for all $x\in X$, we have

\begin{displaymath}
\lim_{n\to +\infty}\frac{1}{n}d(x,\phi^nx)>0.
\end{displaymath} 

\noindent Given a group $G$ acting by isometries on $X$, we denote by $\partial_XG$ the \emph{limit set} of $G$ in $\partial X$, which is defined as the intersection of $\partial X$ with the closure of the orbit of any point in $X$ under the $G$-action. The following theorem, due to Gromov, gives a classification of isometry groups of (possibly nonproper) Gromov hyperbolic spaces. The interested reader will find a sketch of proof in \cite[Proposition 3.1]{CCMT13}.

\begin{theo} (Gromov \cite[Section 8.2]{Gro87})\label{Gromov-1}
Let $X$ be a hyperbolic geodesic metric space, and let $G$ be a group acting by isometries on $X$. Then $G$ is either 
\begin{itemize}
\item \emph{bounded}, i.e. all $G$-orbits in $X$ are bounded; in this case $\partial_X G=\emptyset$, or
\item \emph{horocyclic}, i.e. $G$ is not bounded and contains no loxodromic element; in this case $\partial_X G$ is reduced to one point, or
\item \emph{lineal}, i.e. $G$ contains a loxodromic element, and any two loxodromic elements have the same fixed points in $\partial X$; in this case $\partial_X G$ consists of these two points, or
\item \emph{focal}, i.e. $G$ is not lineal, contains a loxodromic element, and any two loxodromic elements have a common fixed point in $\partial X$; in this case $\partial_X G$ is uncountable and $G$ has a fixed point in $\partial_X G$, or
\item \emph{of general type}, i.e. $G$ contains two loxodromic elements with no common endpoints; in this case $\partial_X G$ is uncountable and $G$ has no finite orbit in $\partial X$. In addition, the group $G$ contains two loxodromic isometries that generate a rank two free subgroup. 
\end{itemize} 
\end{theo}

In particular, we have the following result.

\begin{theo} (Gromov \cite[Section 8.2]{Gro87})\label{Gromov}
Let $X$ be a hyperbolic geodesic metric space, and let $G$ be a group acting by isometries on $X$. If $\partial_X G\neq\emptyset$, and $G$ has no finite orbit in $\partial X$, then $G$ contains a rank two free subgroup generated by two loxodromic isometries.
\end{theo}

\subsection{Outer space}
Let $N\ge 2$. \emph{Outer space} $CV_N$ is defined to be the space of simplicial free, minimal, isometric actions of $F_N$ on simplicial metric trees, up to $F_N$-equivariant homotheties \cite{CV86} (an action of $F_N$ on a tree is \emph{minimal} if there is no proper invariant subtree). We denote by $cv_N$ the \emph{unprojectivized outer space}, in which trees are considered up to equivariant isometries, instead of homotheties. The group $\text{Out}(F_N)$ acts on $CV_N$ and on $cv_N$ on the right by precomposing the actions (one can also consider the $\text{Out}(F_N)$-action on the left by setting $\Phi(T,\rho)=(T,\rho\circ \phi^{-1})$ for all $\Phi\in\text{Out}(F_N)$, where $\rho:F_N\to\text{Isom}(T)$ denotes the action, and $\phi\in\text{Aut}(F_N)$ is any lift of $\Phi$ to $\text{Aut}(F_N)$). 

An \emph{$\mathbb{R}$-tree} is a metric space $(T,d_T)$ in which any two points $x$ and $y$ are joined by a unique arc, which is isometric to a segment of length $d_T(x,y)$. Let $T$ be an \emph{$F_N$-tree}, i.e. an $\mathbb{R}$-tree equipped with an isometric action of $F_N$. For $g\in F_N$, the \emph{translation length} of $g$ in $T$ is defined to be

\begin{displaymath}
||g||_T:=\inf_{x\in T}d_T(x,gx).
\end{displaymath}

\noindent Culler and Morgan have shown in \cite[Theorem 3.7]{CM87} that the map

\begin{displaymath}
\begin{array}{cccc}
i:&cv_N&\to &\mathbb{R}^{F_N}\\
&T&\mapsto & (||g||_T)_{g\in F_N}
\end{array}
\end{displaymath}

\noindent is an embedding, whose image projects to a subspace of $\mathbb{PR}^{F_N}$ with compact closure $\overline{CV_N}$ \cite[Theorem 4.5]{CM87}. Bestvina and Feighn \cite{BF94}, extending results by Cohen and Lustig \cite{CL95}, have characterized the points of this compactification as being the minimal $F_N$-trees with trivial or maximally cyclic arc stabilizers and trivial tripod stabilizers.

\subsection{The free factor complex}\label{sec-ff}

The \emph{free factor complex} $\mathcal{FF}_N$, introduced by Hatcher and Vogtmann in \cite{HV98}, is defined when $N\ge 3$ as the simplicial complex whose vertices are the conjugacy classes of nontrivial proper free factors of $F_N$, and higher dimensional simplices correspond to chains of inclusions of free factors. (When $N=2$, one has to modify this definition by adding an edge between any two complementary free factors to ensure that $\mathcal{FF}_2$ remains connected, and $\mathcal{FF}_2$ is isomorphic to the Farey graph). Gromov hyperbolicity of $\mathcal{FF}_N$ was proved by Bestvina and Feighn \cite{BF12} (see also \cite{KR12} for an alternative proof). There is a natural, coarsely well-defined map $\psi:CV_N\to\mathcal{FF}_N$, that maps any tree $T\in CV_N$ to one of the conjugacy classes of the cyclic free factors of $F_N$ generated by an element of $F_N$ whose axis in $T$ projects to an embedded simple loop in the quotient graph $T/F_N$. The Gromov boundary of $\mathcal{FF}_N$ was determined independently by Bestvina and Reynolds \cite{BR13} and by Hamenstädt \cite{Ham12}. A tree $T\in\partial CV_N$ is \emph{arational} if no proper free factor of $F_N$ acts with dense orbits on its minimal subtree in $T$ (in particular, no proper free factor of $F_N$ is elliptic in $T$). We denote by $\mathcal{AT}$ the subspace of $\partial CV_N$ consisting of arational trees. We define an equivalence relation $\sim$ on $\mathcal{AT}$ by setting $T\sim T'$ whenever $T$ and $T'$ have the same underlying topological tree. 

\begin{theo}\label{boundary-ff} (Bestvina-Reynolds \cite{BR13}, Hamenstädt \cite{Ham12})
There is a unique homeomorphism $\partial\psi:\mathcal{AT}/\sim\to\partial \mathcal{FF}_N$, so that for all $T\in\mathcal{AT}$ and all sequences $(T_n)_{n\in\mathbb{N}}\in CV_N^{\mathbb{N}}$ that converge to $T$, the sequence $(\psi(T_n))_{n\in\mathbb{N}}$ converges to $\partial\psi(T)$.
\end{theo}

Recall from the introduction that an automorphism $\Phi\in\text{Out}(F_N)$ is \emph{fully irreducible} if no nonzero power of $\Phi$ preserves the conjugacy class of any proper free factor of $F_N$. Bestvina and Feighn have characterized elements of $\text{Out}(F_N)$ which act as loxodromic isometries of $\mathcal{FF}_N$.

\begin{theo}(Bestvina-Feighn \cite[Theorem 9.3]{BF12}) \label{loxo-ff}
An outer automorphism $\Phi\in\text{Out}(F_N)$ acts loxodromically on $\mathcal{FF}_N$ if and only if it is fully irreducible.
\end{theo}

\section{An alternative for trees in the boundary of outer space} \label{sec-arational}

Given $T\in\partial{CV_N}\smallsetminus\mathcal{AT}$, the set of conjugacy classes of minimal (with respect to inclusion) proper free factors of $F_N$ which act with dense orbits on their minimal subtree in $T$, but are not elliptic in $T$, is finite \cite[Corollary 7.4 and Proposition 9.2]{Rey12}, and depends $\text{Out}(F_N)$-equivariantly on $T$. We denote it by $\text{Dyn}(T)$.  The following proposition essentially follows from Reynolds' arguments in his proof of \cite[Theorem 1.1]{Rey12}, we provide a sketch for completeness.

\begin{prop} \label{classification}
For all $T\in\partial CV_N\smallsetminus\mathcal{AT}$, either $\text{Dyn}(T)\neq\emptyset$, or there is a nontrivial point stabilizer in $T$ which is contained in a proper free factor of $F_N$.
\end{prop}

In the proof of Proposition \ref{classification}, we will make use of the following well-known fact.

\begin{lemma} (see \cite[Corollary 11.2]{Rey12})\label{stabilizers}
Let $T$ be a simplicial $F_N$-tree, all of whose edge stabilizers are (at most) cyclic. Then every edge stabilizer in $T$ is contained in a proper free factor of $F_N$, and there is at most one conjugacy class of vertex stabilizers in $T$ that is not contained in any proper free factor of $F_N$. 
\end{lemma}

\begin{proof}[Proof of Proposition \ref{classification}]
Let $T\in\partial CV_N\smallsetminus\mathcal{AT}$, and assume that $\text{Dyn}(T)=\emptyset$. First assume that $T$ contains an edge with nontrivial stabilizer. Let $S$ be the simplicial tree obtained by collapsing all vertex trees to points in the decomposition of $T$ as a graph of actions defined in \cite{Lev94}. Then edge stabilizers in $T$ are also edge stabilizers in $S$, and the conclusion follows from Lemma \ref{stabilizers}.

Otherwise, as in the proof of \cite[Proposition 10.3]{Rey12}, we get that if some point stabilizer in $T$ is not contained in any proper free factor of $F_N$, then $T$ is geometric, has dense orbits, and all its minimal components are surfaces (the reader is referred to \cite{BF95,GLP94} for background on geometric $F_N$-trees). Dual to the  decomposition of $T$ into its minimal components is a bipartite simplicial $F_N$-tree $S$ called the \emph{skeleton} of $T$, defined as follows \cite[Section 1.3]{Gui08}. Vertices of $S$ are of two kinds: some correspond to minimal components $Y$ of $T$, and the others correspond to points $x\in T$ belonging to the intersection of two distinct minimal components. There is an edge from the vertex associated to $x$ to the vertex associated to $Y$ whenever $x\in Y$. In particular, point stabilizers in $S$ are either point stabilizers in $T$, or groups acting with dense orbits on their minimal subtree in $T$. In a minimal surface component, all point stabilizers are cyclic, so $S$ is a simplicial $F_N$-tree with (at most) cyclic edge stabilizers. If $S$ is nontrivial, Lemma \ref{stabilizers} implies that either a point stabilizer in $T$ is contained in a proper free factor, or some subgroup of $F_N$ is contained in a proper free factor $F$ of $F_N$ and acts with dense orbits on its minimal subtree in $T$. In the latter case, by decomposing the $F$-action on the $F$-minimal subtree of $T$ as a graph of actions with trivial arc stabilizers \cite{Lev94}, we get that $\text{Dyn}(T)\neq\emptyset$, which has been excluded. If $S$ is reduced to a point, then $T$ is minimal and dual to a surface with at least two boundary curves (otherwise $T$ would be arational by \cite[Theorem 1.1]{Rey12}). Any of these curves yields the desired point stabilizer in $T$.
\end{proof}

\section{Nonelementary subgroups of $\text{Out}(F_N)$}

A subgroup $H\subseteq\text{Out}(F_N)$ is \emph{nonelementary} if it does not preserve any finite set of $\mathcal{FF}_N\cup\partial\mathcal{FF}_N$. In this section, we will prove Theorem \ref{HM} for nonelementary subgroups of $\text{Out}(F_N)$.

\begin{theo}\label{nonelementary}
Every nonelementary subgroup of $\text{Out}(F_N)$ contains a rank two free subgroup, generated by two fully irreducible automorphisms. 
\end{theo}

\paragraph*{Stationary measures on $\partial{CV_N}$.}

Our proof of Theorem \ref{nonelementary} is based on techniques that originally arose in the study of random walks on groups. All topological spaces will be equipped with their Borel $\sigma$-algebra. Let $\mu$ be a probability measure on $\text{Out}(F_N)$. A probability measure $\nu$ on $\overline{CV_N}$ is \emph{$\mu$-stationary} if $\mu\ast\nu=\nu$, i.e. for all $\nu$-measurable subsets $E\subseteq \overline{CV_N}$, we have

\begin{displaymath}
\nu(E)=\sum_{\Phi\in\text{Out}(F_N)}\mu(\Phi)\nu(\Phi^{-1}E).
\end{displaymath}

\noindent Our first goal will be to prove the following fact.

\begin{prop}\label{stationary}
Let $\mu$ be a probability measure on $\text{Out}(F_N)$, whose support generates a nonelementary subgroup of $\text{Out}(F_N)$. Then every $\mu$-stationary probability measure on $\overline{CV_N}$ is supported on $\mathcal{AT}$. 
\end{prop}

We will make use of the following classical lemma, whose proof is based on a maximum principle argument (we provide a sketch for completeness). We denote by $gr(\mu)$ the subgroup of $\text{Out}(F_N)$ generated by the support of the measure $\mu$.

\begin{lemma} \label{disjoint-translations} (Ballmann \cite{Bal89}, Woess \cite[Lemma 3.4]{Woe89}, Kaimanovich-Masur \cite[Lemma 2.2.2]{KM96})
Let $\mu$ be a probability measure on a countable group $G$, and let $\nu$ be a $\mu$-stationary probability measure on a $G$-space $X$. Let $D$ be a countable $G$-set, and let $\Theta:X\to D$ be a measurable $G$-equivariant map. If $E\subseteq X$ is a $G$-invariant measurable subset of $X$ satisfying $\nu(E)>0$, then $\Theta(E)$ contains a finite $gr(\mu)$-orbit.
\end{lemma}

\begin{proof}
Let $\widetilde{\nu}$ be the probability measure on $D$ defined by setting $\widetilde{\nu}(Y):=\nu(\Theta^{-1}(Y))$ for all subsets $Y\subseteq D$. It follows from $\mu$-stationarity of $\nu$ and $G$-equivariance of $\Theta$ that $\widetilde{\nu}$ is $\mu$-stationary. Let $M\subseteq\Theta(E)$ denote the set consisting of all $x\in \Theta(E)$ such that $\widetilde{\nu}(x)$ is maximal (and in particular positive). Since $\widetilde{\nu}$ is a probability measure, the set $M$ is finite and nonempty. For all $x\in M$, we have

\begin{displaymath}
\widetilde{\nu}(x)=\sum_{g\in G}\mu(g)\widetilde{\nu}(g^{-1}x)\le \widetilde{\nu}(x)\sum_{g\in G}\mu(g)=\widetilde{\nu}(x),
\end{displaymath}

\noindent which implies that for all $g\in G$ belonging to the support of $\mu$, we have $\widetilde{\nu}(g^{-1}x)=\widetilde{\nu}(x)$. Therefore, the set $M$ is invariant under the semigroup generated by the support of $\check{\mu}$ (where $\check{\mu}(g):=\mu(g^{-1})$). As $M$ is finite, this implies that $M$ is $\text{gr}(\mu)$-invariant, so it contains a finite $\text{gr}(\mu)$-orbit. 
\end{proof}

We now define an $\text{Out}(F_N)$-equivariant map $\Theta$ from $\overline{CV_N}$ to the (countable) set $D$ of finite collections of conjugacy classes of proper free factors of $F_N$. Given a tree $T\in CV_N$, we define $\text{Loop}(T)$ to be the finite collection of conjugacy classes of elements of $F_N$ whose axes in $T$ project to an embedded simple loop in the quotient graph $T/F_N$ (these may be viewed as cyclic free factors of $F_N$). Given $T\in\overline{CV_N}$, the set of conjugacy classes of point stabilizers in $T$ is finite \cite{Jia91}. Every point stabilizer is contained in a unique minimal (possibly non proper) free factor of $F_N$, defined as the intersection of all free factors of $F_N$ containing it (the intersection of a family of free factors of $F_N$ is again a free factor). We let $\text{Per}(T)$ be the (possibly empty) finite set of conjugacy classes of proper free factors of $F_N$ that arise in this way, and we set

\begin{displaymath}
\Theta(T):=\left\{
\begin{array}{ll}
\emptyset &\text{~if~} T\in\mathcal{AT}\\
\text{Loop}(T) &\text{~if~} T\in CV_N\\
\text{Dyn}(T)\cup\text{Per}(T) &\text{~if~} T\in\partial CV_N\smallsetminus\mathcal{AT}
\end{array}\right..
\end{displaymath} 

\noindent Proposition \ref{classification} implies that $\Theta(T)=\emptyset$ if and only if $T\in\mathcal{AT}$.

\begin{lemma}\label{Theta-measurable}
The set $\mathcal{AT}$ is measurable, and $\Theta$ is measurable. 
\end{lemma}

We postpone the proof of Lemma \ref{Theta-measurable} to the next paragraph and first explain how to deduce Proposition \ref{stationary}.

\begin{proof}[Proof of Proposition \ref{stationary}]
Nonelementarity of $gr(\mu)$ implies that the only finite $gr(\mu)$-orbit in $D$ is the orbit of the empty set. Therefore, since $\Theta(T)\neq\emptyset$ as soon as $T\in\overline{CV_N}\smallsetminus\mathcal{AT}$ (Proposition \ref{classification}), the set $\Theta(\overline{CV_N}\smallsetminus\mathcal{AT})$ contains no finite $gr(\mu)$-orbit. Proposition \ref{stationary} then follows from Proposition \ref{disjoint-translations}.
\end{proof}

\paragraph*{Measurability of $\Theta$.}
Given a finitely generated subgroup $F$ of $F_N$, we denote by $\mathcal{P}(F)$ the set of trees $T\in\overline{CV_N}$ in which $F$ is elliptic, by $\mathcal{E}(F)$ the set of trees $T\in\overline{CV_N}$ in which $F$ fixes an edge, and by $\mathcal{D}(F)$ the set of trees $T\in\overline{CV_N}$ whose $F$-minimal subtree is a nontrivial $F$-tree with dense orbits.

\begin{lemma} \label{measurability-reducing}
For all finitely generated subgroups $F\subseteq F_N$, the sets $\mathcal{P}(F)$, $\mathcal{E}(F)$ and $\mathcal{D}(F)$ are measurable.
\end{lemma}

\begin{proof}
Let $F$ be a finitely generated subgroup of $F_N$. Let $s:\overline{CV_N}\to\overline{cv_N}$ be a continuous section. We have

\begin{displaymath}
\mathcal{P}(F)=\bigcap_{w\in F}\{T\in\overline{CV_N}|||w||_{s(T)}=0\},
\end{displaymath}

\noindent so $\mathcal{P}(F)$ is measurable. An element $g\in F_N$ fixes an arc in a tree $T\in\overline{CV_N}$ if and only if it is elliptic, and there exist two hyperbolic isometries $h$ and $h'$ of $T$ whose translation axes both meet the fixed point set of $g$ but are disjoint from each other. These conditions can be expressed in terms of translation length functions: they amount to requiring that $||gh||_{s(T)}\le||h||_{s(T)}$ and $||gh'||_{s(T)}\le||h'||_{s(T)}$ and $||hh'||_{s(T)}>||h||_{s(T)}+||h'||_{s(T)}$, see \cite[1.5]{CM87}. So $\mathcal{E}(F)$ is a measurable set, too.

The $F$-minimal subtree of a tree $T\in\overline{CV_N}$ has dense orbits if and only if for all $n\in\mathbb{N}$, there exists a free basis $\{s_1,\dots,s_k\}$ of $F$ so that for all $i,j\in\{1,\dots,k\}$, we have $||s_i||_{s(T)}\le\frac{1}{n}$ and $||s_is_j||_{s(T)}\le\frac{1}{n}$. This implies that the set $\text{Dense}(F)$ consisting of those trees in $\overline{CV_N}$ whose $F$-minimal subtree has dense orbits is measurable. Therefore $\mathcal{D}(F)={\text{Dense}}(F)\cap^c\mathcal{P}(F)$ is also measurable.  
\end{proof}

\begin{proof}[Proof of Lemma \ref{Theta-measurable}]
Measurability of $\mathcal{AT}$ follows from Lemma \ref{measurability-reducing}. For all $T\in\overline{CV_N}$, the set $\text{Dyn}(T)$ consists of conjugacy classes of minimal free factors of $F_N$ that act with dense orbits on their minimal subtree in $T$ but are not elliptic, so measurability of the map $T\mapsto \text{Dyn}(T)$ follows from measurability of $\mathcal{D}(F)$ for all finitely generated subgroups $F$ of $F_N$. Point stabilizers in $T$ are either maximal among elliptic subgroups, or fix an arc in $T$. Therefore, since $\mathcal{P}(F)$ and $\mathcal{E}(F)$ are measurable for all finitely generated subgroups $F$ of $F_N$, the set of conjugacy classes of point stabilizers in a tree $T\in\overline{CV_N}$ depends measurably on $T$. Measurability of $T\mapsto\text{Per}(T)$ follows from this observation. As open simplices in $CV_N$ are also measurable, measurability of $\Theta$ follows.  
\end{proof}

\paragraph*{End of the proof of Theorem \ref{nonelementary}.}

\begin{prop}\label{nonempty-limit-set}
Let $H\subseteq\text{Out}(F_N)$ be a nonelementary subgroup of $\text{Out}(F_N)$. Then the $H$-orbit of any point $x_0\in CV_N$ has a limit point in $\mathcal{AT}$.
\end{prop}

\begin{proof}
Let $\mu$ be a probability measure on $\text{Out}(F_N)$ whose support generates $H$. Since $\overline{CV_N}$ is compact, the sequence of convolutions $(\mu^{\ast n}\ast\delta_{x_0})_{n\in\mathbb{N}}$ has a weak-$\ast$ limit point $\nu$, which is a $\mu$-stationary measure on $\overline{CV_N}$.  We have $\nu(\overline{Hx_0})=1$, where $Hx_0$ denotes the $H$-orbit of $x_0$ in $CV_N$, and Proposition \ref{stationary} implies that $\nu(\mathcal{AT})=1$. This shows that $\overline{Hx_0}\cap\mathcal{AT}$ is nonempty.
\end{proof}

As a consequence of Theorem \ref{boundary-ff} and Proposition \ref{nonempty-limit-set}, we get the following fact.

\begin{cor}\label{limit-set-fz}
Let $H\subseteq\text{Out}(F_N)$ be a nonelementary subgroup of $\text{Out}(F_N)$. Then the $H$-orbit of any point in $\mathcal{FF}_N$ has a limit point in $\partial \mathcal{FF}_N$.
\qed
\end{cor}

\begin{proof}[Proof of Theorem \ref{nonelementary}]
Let $H$ be a nonelementary subgroup of $\text{Out}(F_N)$. Corollary \ref{limit-set-fz} shows that the $H$-orbit of any point in $\mathcal{FF}_N$ has a limit point in $\partial \mathcal{FF}_N$. As $H$ has no finite orbit in $\partial\mathcal{FF}_N$, Theorem \ref{Gromov} shows that $H$ contains two loxodromic isometries which generate a free group of rank two. Theorem \ref{nonelementary} then follows from the fact that elements of $\text{Out}(F_N)$ that act loxodromically on $\mathcal{FF}_N$ are fully irreducible (Theorem \ref{loxo-ff}).  
\end{proof}

\section{Proof of Theorem \ref{HM}}

Let $H$ be a subgroup of $\text{Out}(F_N)$. If $H$ is nonelementary, then the claim follows from Theorem \ref{nonelementary}. Otherwise, either $H$ fixes a finite subset of conjugacy classes of proper free factors (in which case a finite index subgroup of $H$ fixes the conjugacy class of a proper free factor of $F_N$), or $H$ virtually fixes a point in $\partial \mathcal{FF}_N$. The set of trees in $\partial CV_N$ that project to this point is a finite-dimensional simplex in $\partial CV_N$ by \cite[Corollary 5.4]{Gui00}, and $H$ fixes the finite subset of extremal points of this simplex. Up to passing to a finite index subgroup again, we can assume that $H$ fixes an arational tree $T\in\partial CV_N$. By Reynolds' characterization of arational trees \cite[Theorem 1.1]{Rey12}, either $T$ is free, or else $T$ is dual to an arational measured lamination on a surface $S$ with one boundary component. In the first case, it follows from \cite[Theorem 1.1]{KL11} that $H$ is virtually cyclic, virtually generated by an automorphism $\Phi\in\text{Out}(F_N)$, and in this case $\Phi$ is fully irreducible, otherwise $H$ would virtually fix the conjugacy class of a proper free factor of $F_N$. In the second case, all automorphisms in $H$ can be realized as diffeomorphisms of $S$ \cite[Theorem 4.1]{BH92}. So $H$ is a subgroup of the mapping class group of $S$, and the claim follows from the analogous classical statement that stabilizers of arational measured foliations are virtually cyclic \cite[Proposition 2.2]{McP89}.
\qed

\bibliographystyle{amsplain}
\bibliography{/Users/Camille/Documents/Bibliographie}

\end{document}